\documentclass[11pt]{amsart}

\usepackage{amscd,amsmath,amssymb,fancyhdr,color}
\usepackage[utf8]{inputenc}
\usepackage{amsfonts}
\usepackage{amsthm}
\usepackage{accents}
\usepackage{graphicx}
\usepackage{float}
\usepackage{subcaption}
\usepackage{verbatim}
\usepackage{indentfirst}
\usepackage{dsfont}
\usepackage{tikz}
\usepackage{tikz-cd}
\usetikzlibrary{matrix}
\usepackage[all]{xy}
\usepackage{enumerate}
\usepackage{extarrows}
\usepackage{csquotes}

\usepackage{scalerel,stackengine}
\stackMath
\newcommand\reallywidehat[1]{%
	\savestack{\tmpbox}{\stretchto{%
			\scaleto{%
				\scalerel*[\widthof{\ensuremath{#1}}]{\kern-.6pt\bigwedge\kern-.6pt}%
				{\rule[-\textheight/2]{1ex}{\textheight}}
			}{\textheight}%
		}{0.5ex}}%
	\stackon[1pt]{#1}{\tmpbox}%
}

\usepackage{faktor}

\usepackage{BOONDOX-uprscr}

\usepackage[backref=page]{hyperref}
\renewcommand*{\backref}[1]{}
\renewcommand*{\backrefalt}[4]{%
	\ifcase #1 (Not cited.)%
	\or        (Cited on page~#2.)%
	\else      (Cited on pages~#2.)%
	\fi}

\hypersetup{
	colorlinks   = true,
	citecolor    = magenta
}

\newcommand{\K}{K\"ahler}

\DeclareMathOperator{\reg}{reg}

\DeclareMathOperator{\FS}{FS}

\numberwithin{equation}{section}

\def\eqref#1{(\ref{#1})}

\newcommand{\Z}{{\mathbb Z}}
\newcommand{\C}{{\mathbb C}}
\newcommand{\R}{{\mathbb R}}

\newcommand{\ei}{\textup{i}}

\newcommand{\del}{\partial}
\newcommand{\delb}{\overline{\partial}}

\def\1{\sqrt{-1}\:}

\newcommand{\cntrct}                
{\hspace{2pt}\raisebox{1pt}{\text{$\lrcorner$}}\hspace{2pt}}

\newcommand{\codim}{\operatorname{codim}}
\renewcommand{\dim}{\operatorname{dim}}

\newcommand{\eg}{{\em e.g. }}

\renewcommand{\to}{\longrightarrow}

\newcounter{Mycounter}[section]
\newcounter{lemma}[section]
\newcounter{claim}[section]
\newcounter{sublemma}[section]
\newcounter{corollary}[section]
\newcounter{theorem}[section]
\newcounter{conjecture}[section]
\newcounter{proposition}[section]
\newcounter{definition}[section]
\newcounter{example}[section]
\newcounter{remark}[section]
\newcounter{problem}[section]
\newcounter{question}[section]
\makeatletter

\@addtoreset{equation}{section}

\@addtoreset{footnote}{section}

\makeatother

\usetikzlibrary{arrows,chains,matrix,positioning,scopes}

\makeatletter
\tikzset{join/.code=\tikzset{after node path={%
			\ifx\tikzchainprevious\pgfutil@empty\else(\tikzchainprevious)%
			edge[every join]#1(\tikzchaincurrent)\fi}}}
\makeatother

\tikzset{>=stealth',every on chain/.append style={join},
	every join/.style={->}}

\makeatletter
\newtheorem*{rep@theorem}{\rep@title}
\newcommand{\newreptheorem}[2]{%
	\newenvironment{rep#1}[1]{%
		\def\rep@title{\ref{##1}}%
		\begin{rep@theorem}}%
		{\end{rep@theorem}}}
\makeatother

\newreptheorem{theorem}{Theorem}

\begin{document}
	
	\newpage
	
	\title[Vaisman's theorem and local reducibility]{Vaisman's theorem and local reducibility}

\author{Ovidiu Preda}
\address{Ovidiu Preda \newline
	\textsc{\indent University of Bucharest, Faculty of Mathematics and Computer Science\newline 
		\indent 14 Academiei Str., Bucharest, Romania\newline
		\indent \indent and\newline
		\indent Institute of Mathematics ``Simion Stoilow'' of the Romanian Academy\newline 
		\indent 21 Calea Grivitei Street, 010702, Bucharest, Romania}}
\email{ovidiu.preda@fmi.unibuc.ro; ovidiu.preda@imar.ro}

\author{Miron Stanciu}
\address{Miron Stanciu \newline
	\textsc{\indent University of Bucharest, Faculty of Mathematics and Computer Science\newline 
		\indent 14 Academiei Str., Bucharest, Romania\newline
		\indent \indent and \newline
		\indent Institute of Mathematics ``Simion Stoilow'' of the Romanian Academy\newline 
		\indent 21 Calea Grivitei Street, 010702, Bucharest, Romania}}
\email{miron.stanciu@fmi.unibuc.ro; miron.stanciu@imar.ro}

\thanks{Both authors were partially supported by a grant of Ministry of Research and Innovation, CNCS - UEFISCDI, project no.
	PN-III-P1-1.1-TE-2021-0228, within PNCDI III. \\\\[.1cm]
}

\date{\today}

	\begin{abstract}
		As proven in a celebrated theorem due to Vaisman, pure locally conformally \K \ metrics do not exist on compact \K\ manifolds. In a previous paper, we extended this result to the singular setting, more precisely to \K\ spaces which are locally irreducible. Without the additional assumption of local irreducibility, there are counterexamples for which Vaisman's theorem does not hold. In this article, we give a much broader sufficient condition under which Vaisman's theorem still holds for compact \K\ spaces which are locally reducible.
	\end{abstract}
	
	\maketitle
	
	\hypersetup{linkcolor=blue}
	\tableofcontents

	\section{Introduction}

    A key challenge in complex geometry is identifying suitable Hermitian metrics whose existence leads to nice geometric or topological results, but that also occur on a large enough class of spaces. While the previous century has seen extensive research of \K \ metrics on manifolds, their existence imposes significant analytic and topological restrictions, so in order to satisfy the second requirement, much work has been done recently to weaken the \K \ condition and so look at interesting classes of special non-\K \ metrics. Among the most studied such metrics are called \textit{locally conformally \K} introduced by Vaisman in \cite{vaisman1976}.

    Locally conformally \K \ (\textit{lcK} for short) manifolds are, as the name implies, complex manifolds which admit a $(1, 1)$-form $\omega$, such that, locally, there exist smooth functions $f$, for which $e^{-f}\omega$ is \K. This immediately implies that the differentials $df$ agree on the intersections of these sets and glue up to a closed $1$-form $\theta$, so the lcK condition can be expressed equivalently as $d\omega=\theta\wedge\omega$. The closed $1$-form $\theta$ is called the Lee form. By definition, $\theta$ is exact if and only if $f$ can be defined globally. In this case, $\omega$ is called \textit{globally conformally K\" ahler} (\textit{gcK}). If $\omega$ is lcK, but not gcK, we call it \textit{pure lcK}. Later, Vaisman \cite{vaisman1980} proved that on a compact complex manifold, pure lcK and \K\ metrics (with respect to the same complex structure) cannot coexist. For an extensive and up-to-date survey of lcK smooth geometry, see \cite{OV_book}.
 
	Just as Grauert extended the definition of \K\ metrics to complex analytic spaces in \cite{grauert}, one can also define in a similar manner lcK metrics on complex spaces with singularities (see \cite{PS21}). Spaces carrying such metrics may occur naturally even if one is only considering problems of a smooth nature, for example as a non-generic fibre of a bundle whose total space is lcK, so or continuing goal over the last few years was generalizing as many results as possible to this much broader class. Among the known facts about lcK manifolds, Vaisman's theorem is perhaps the most powerful. However, in this setting one no longer has access to many basic tools that were very useful in the smooth case, for instance Hodge theory or the $\del \delb$-Lemma; every proof of Vaisman's theorem in the smooth case makes use of one of these results. 
 
    In \cite{PS23}, making use of the Hironaka desingularization procedure, we were able to show that Vaisman's theorem still holds under the additional assumption that the space is locally irreducible, and also give a counterexamples to show that at least for some locally reducible spaces, it does not. In this article, we continue this work to a much finer analysis of the analytic and topological properties a complex space must have in order for Vaisman's theorem to fail; in our main theorem, we improve our previous result in \cite{PS23} to find a much larger class of complex spaces for which Vaisman's theorem holds.

    \medskip
	
	The paper is organized as follows: in Section \ref{sec:prelim}, we recall the basic definitons, previous results of ours, as well as a few additional facts that we will need later. In Section \ref{sec:main}, in addition to a few auxiliary results, we state and prove our characterization of the class complex \K \ spaces on which Vaisman's theorem still works. We end with Section \ref{sec:examples}, in which we give two classes of examples: one with a space that is not locally irreducible, but satisfies our new conditions, so Vaisman's Theorem applies, and one giving a typical scenario in which Vaisman's theorem still fails to work.
	
	\section{Preliminaries}\label{sec:prelim}
	
	Firstly, we recall the definitions for \K\ and lcK metrics on complex analytic spaces.
	
		\begin{definition}\label{K+lcK}
		Let $X$ be a complex analytic space. 
		\begin{enumerate}
			\item[\textbf{(K)}] A \textit{K\" ahler metric} on $X$ is the equivalence class $\reallywidehat{(U_i,\varphi_i)_{i\in I}}$ of a family such that $(U_i)_{i\in I}$ is an open cover of $X$, $\varphi_i:U_i\rightarrow \mathbb{R}$ is $\mathcal{C}^\infty$ and strictly psh, and $\ei\partial\overline{\partial}\varphi_i=\ei\partial\overline{\partial}\varphi_j$ on $U_i\cap U_j\cap X_{\reg}$, for every $i,j\in I$. 
			Two such families are equivalent if their union verifies the compatibility condition on the intersections, described above.
			
			\item[\textbf{(lcK)}] An \textit{lcK metric} on $X$ is the equivalence class $\reallywidehat{(U_i,\varphi_i,f_i)_{i\in I}}$ of a family such that $(U_i)_{i\in I}$ is an open cover of $X$, $\varphi_i:U_i\rightarrow \mathbb{R}$ is $\mathcal{C}^\infty$ and strictly psh, $f_i:U_i\rightarrow\mathbb{R}$ is smooth, and $\ei e^{f_i}\partial\overline{\partial}\varphi_i=\ei e^{f_j}\partial\overline{\partial}\varphi_j$ on $U_i\cap U_j\cap X_{\reg}$, for every $i,j\in I$. 
			As before, two such families are equivalent if their union verifies the compatibility condition on the intersections.
		\end{enumerate}
	\end{definition}
	
	Since for lcK forms on singular spaces we also want to define its associated Lee form, we have the following definition of a similar nature:
	
	\begin{definition}\label{TC-1-form-definition}
		\begin{itemize}
			\item Let $X$ be a topological space and consider $(U_i,f_i)_{i\in I}$, consisting of an open cover $(U_i)_{i\in I}$ of $X$ and a family of continuous functions $f_i:U_i\rightarrow\mathbb{R}$ such that $f_i-f_j$ is locally constant on $U_i\cap U_j$, for all $i,j\in I$. The class  
			\[
			\theta =\reallywidehat{(U_i,f_i)_{i\in I}}\in  \check{\mathrm{H}}^0\left(X,\faktor{\mathscr{C}}{\underline{\R}}\right)
			\]
			is called a \textit{topologically closed 1-form} (\textit{TC1-form}). 
			\item We say that a TC1-form $\theta$ is \textit{exact} if $\theta = \widehat{(X, f)}$ for a continuous function $f:X \rightarrow \mathbb{R}$. In this case, we make the notation $\theta = df$.
			\item 	Let $\omega=\reallywidehat{(U_i,\varphi_i,f_i)_{i\in I}}$ be an lcK metric on a complex space $X$. Then, the TC1-form $\theta=\reallywidehat{(U_i,f_i)_{i\in I}}$ is called the \textit{Lee form} of $\omega$. If $\theta$ is exact, then $\omega$ is called \textit{globally conformally K\" ahler (gcK)}.
		\end{itemize}
	\end{definition}
	
	\vspace{5pt}
	
	An essential ingredient for our main result is Vaisman's theorem for locally irreducible, compact \K\ spaces \cite[Thm.4.4]{PS23}. Since it stays true with exactly the same proof even for wlcK spaces (compare with \cite[Lemma 2.5]{APV23}, the statement below includes this case.
	
	\begin{theorem}\label{Vaisman_wlck}
		Let $(X,\omega,\theta)$ be a compact, locally irreducible, (w)lcK space. If $X$ admits a \K\ metric, then $(X,\omega,\theta)$ is (w)gcK.
	\end{theorem}
	
	\vspace{5pt}
	
	We also need the next result on the existence of \K\ metrics on spaces which project with discrete fibers on \K\ spaces:
	
	\begin{theorem}(\cite[Thm.1]{vaj96})\label{Vajaitu_discrete_fibers}
		Let $f:X\rightarrow Y$ be a holomorphic map of complex spaces with discrete fibers. If $Y$ is K\" ahlerian, then $X$ is K\" ahlerian too.
	\end{theorem}
	
	\medskip
	
	The following theorem by Henri Cartan \cite[Main Theorem]{HCartan60} gives a sufficient condition under which the quotient of a complex space by a proper equivalence relation also has a structure of complex space. It will be needed for constructing the examples in section \ref{sec:examples}.
	
	\begin{theorem}\label{Cartan_criterion}
		Consider a proper equivalence relation $R$ on a complex space $X$, with quotient map $\pi:X\rightarrow X/R$. In order that the ringed space $X/R$ be a complex space, it suffices that each point of $X/R$ has an open neighborhood $V$ such that the $R$-invariant holomorphic maps $\pi^{-1}(V)\rightarrow Z$ ($Z$ being a complex space) separate the equivalence classes in $\pi^{-1}(V)$.
	\end{theorem}
	
	\vspace{10pt}
	
	In the same constructions, we will make use of \cite[Prop. 2.1 \& Cor.2.2]{CM85} about pushing forward (strictly) psh functions:
	
	\begin{proposition}\label{PROP-Coltoiu-Mihalache-s-psh}
		Let $X, Y$ be complex spaces and $p:X\rightarrow Y$ a proper, surjective, holomorphic map. Let $\phi:Y\to [-\infty,\infty)$ be an upper semicontinous function such that $\phi\circ p$ is (strictly) psh on $X$. Then, $\phi$ is (strictly) psh on $Y$.
	\end{proposition}
	
	\vspace{10pt}
	
	An immediate consequence to this is \cite[Cor.2.3]{CM85}:
	
	\begin{corollary}\label{COR-Coltoiu-Mihalache-s-psh}
		Let $X$ be a complex space and $\phi:X\rightarrow [-\infty,\infty)$ an upper semicontinuous function. Then, $\varphi$ is (strictly) psh on $X$ iff restricted to any irreducible component of $X$ is (strictly) psh.
	\end{corollary}
	
	\section{The main result}\label{sec:main}

	Let $X$ be a complex analytic space of dimension $\dim X=n$ and consider 
	$$X=Y_0\cup Y_1\cup\ldots \cup Y_{n}$$
	its stratification (for the construction of the stratification of a complex space, see \eg \cite[Chapter II, Prop.5.6]{Demailly_book}). $Y_0$ is a discrete set, $Y_{k-1}\subset Y_{k}$ and $Y_{k}\setminus Y_{k-1}$ is a smooth $k$-dimensional complex manifold for each $1 \leq k\leq n$. The sets $Y_{k}\setminus Y_{k-1}$ are called the strata; they might not be connected, and some might be empty for indices $k<n$. 

	To help the structure of the proof of the main theorem, we begin with two lemmas. 
	
	\begin{lemma}\label{key_lemma}
		Let $\pi:\widehat{Y}\rightarrow Y$ be a finite covering with $p$ sheets of topological spaces and $\theta$ be a TC1-form on $Y$ such that  its pull-back $\pi^*\theta$ is exact. Then, $\theta$ is exact.
	\end{lemma}
	\begin{proof}
		Let $\gamma$ be a closed curve in $Y$. Then, there exists a closed curve $\widehat{\gamma}$ in $\widehat{Y}$ such that the restriction $\pi:\widehat{\gamma}\rightarrow \gamma$ is a covering with $q$ sheets, $q\leq p$. Thus, with the definition of integral given in \cite[Section 2]{PS21}, we have
		$$\int_{\widehat{\gamma}}\pi^*\theta = q\int_\gamma\theta.$$
		However, since $\pi^*\theta$ is exact, we have $\int_{\widehat{\gamma}}\pi^*\theta=0$, hence $\int_\gamma \theta=0$. As this is true for every closed curve $\gamma$ in $Y$, we obtain that $\theta$ is exact.
	\end{proof}

	\vspace{10pt}
	
		\begin{lemma}\label{key_lemma_2}
		Let $\pi:\widehat{Y}\rightarrow Y$ be a ramified covering with $p$ sheets of complex spaces such that $Y$ has no singularities, and let $\theta$ be a TC1-form on $Y$ such that  its pull-back $\pi^*\theta$ is exact. Then, $\theta$ is exact.
	\end{lemma}
	\begin{proof}
		Denote by $S\subset Y$ the ramification locus. Since $\codim_Y S\geq 1$ and $Y$ is a complex manifold, $Y\setminus S$ is locally connected around the points of $S$. Let $\gamma \subset Y$ be a closed curve. Then, $\gamma$ is homotopically equivalent to a closed curve $\eta \subset Y\setminus S$. Hence, $\int_\gamma \theta = \int_\eta \theta$. Now, applying \ref{key_lemma} for the unramified covering 
		$$\pi:\widehat{Y}\setminus \pi^{-1}(S)\rightarrow Y\setminus S,$$
		we get that $\theta_{\restriction Y\setminus S}$ is exact, thus $\int_\eta \theta=0$, which in turn gives $\int_\gamma \theta=0$. Finally, since this is true for any closed curve $\gamma$, we get that $\theta$ is exact. 
	\end{proof}
	
	\vspace{10pt}
	
	Now, we can prove the main theorem.
	
	\begin{theorem}\label{MAIN_THM}
		Let $(X,\omega_0)$ be a globally irreducible, compact \K\ space. Denote by $\pi:\widehat{X}\to X$ its normalization and consider $$X=Y_0\cup Y_1\cup\ldots \cup Y_{n}$$ 
		the stratification of $X$. For every $0\leq k\leq n$, let
		$$Y_k=\bigcup_{j}Y_{k,j}$$ 
		be the decomposition into globally irreducible components. If for every $k$ and $j$, the set $\pi^{-1}(Y_{k,j})$ is connected, then $X$ does not admit pure lcK metrics. 

	\end{theorem}
	\begin{proof} First, we assume that for every $k$ and $j$, the set $\pi^{-1}(Y_{k,j})$ is connected. We know that $(X,\omega_0)$ is a \K\ space, and assume also that we have another metric $\omega$ on $X$ such that $(X,\omega,\theta)$ is an lcK space. Then, since the normalization mapping $\pi$ has finite fibers, by \ref{Vajaitu_discrete_fibers}, $\widehat{X}$ admits a \K\ metric, and pulling back the metric $\omega$, we obtain that $(\widehat{X}, \pi^*\omega, \pi^*\theta)$ is a wlcK space. \ref{Vaisman_wlck} then yields that $\pi^*\omega$ is wgcK, thus $\pi^*\theta=d\widehat{f}$. 
		
	Next, we show that $\widehat{f}$ is constant on the fibers of $\pi$. We do this by induction on $k$, which is the dimension of the stratum. For the verification step, take $x\in Y_0$. By our assumption, $\pi^{-1}(x)$ is connected, hence it consists of only one point and $\widehat{f}$ is trivially constant on this fiber. For the induction step, assume that $\widehat{f}$ is constant on the fibers above $Y_{k-1}$. Denote 
	$$\pi^{-1}(Y_{k,j})=\widehat{Y}_{k,j}=\bigcup_{l}\widehat{Y}_{k,j,l}$$ 
	the decomposition into irreducible components of $\widehat{Y}_{k,j}$.
	$\pi$ is holomorphic and finite, hence the direct image of an analytic set is also analytic, and it preserves the dimension when taking direct or inverse images of analytic sets. 
	If $\dim \widehat{Y}_{k,j,l}<k$, then $\dim\pi(\widehat{Y}_{k,j,l})<k$. By the construction of the stratification, this means that $\pi(\widehat{Y}_{k,j,l})\subset Y_{k-1}$, thus $\widehat{Y}_{k,j,l} \setminus \pi^{-1}(Y_{k-1})=\emptyset$, and there is nothing to prove about $\widehat{f}$ on the fibers of $\pi$ restricted to this set, so we are left with studying the case in which there exists an index $l$ such that $\dim \widehat{Y}_{k,j,l}=k$.
	$\pi(\widehat{Y}_{k,j,l})$ is analytic in $Y_{k,j}$ and of the same dimension $k$. Further, $Y_{k,j}$ being globally irreducible yields $\pi(\widehat{Y}_{k,j,l})=Y_{k,j}$. Also, 
	$$\pi:\widehat{Y}_{k,j,l} \setminus \pi^{-1}(Y_{k-1})\rightarrow Y_{k,j}\setminus Y_{k-1}$$ 
	is a (possibly ramified) covering with finite number of sheets of $Y_{k,j}\setminus Y_{k-1}$, which is a space with no singularities. Applying \ref{key_lemma_2}, we get that $\theta_{\restriction Y_{k,j}\setminus Y_{k-1}}=df$. Then, we have
	$$d\widehat{f}=\pi^*\theta=\pi^*df=d(f\circ \pi) \hspace{10pt} \text{on} \hspace{10pt} \widehat{Y}_{k,j}\setminus \pi^{-1}(Y_{k-1}),$$ 
	hence $\widehat{f}-f\circ\pi$ is constant on the connected set $\widehat{Y}_{k,j,l}\setminus \pi^{-1}(Y_{k-1})$. By continuity, $\widehat{f}_{\restriction \widehat{Y}_{k,j,l}}$ is constant on the fibers, therefore $f$ extends continuously to $Y_{k,j}$. Recalling now that $\widehat{Y}_{k,j}$ is connected, we obtain that $\widehat{f}-f\circ\pi$ is constant on $\widehat{Y}_{k,j}$. Since this is true for any $j$, we get that $\widehat{f}$ is constant on the fibers above $Y_k$, ending the proof of the induction step. 
	
	Finally, we know that $\pi^*\theta=d\widehat{f}$ and $\widehat{f}$ is constant on the fibers of $\pi$, hence $\theta=df$, where $f$ is any function such that $\widehat{f}-f\circ\pi=c\in\mathbb{R}$ on $X$, which means that $(X,\omega,\theta)$ is gcK.
    \end{proof}

	\section{Some examples}\label{sec:examples}
	
	In this section, we present two examples of locally reducible complex spaces: one for which Vaisman's theorem holds, and one for which it fails. They are both obtained by identifying two biholomorphic submanifolds in a complex manifold.
	
	\begin{example}
		We consider the projective space $\mathbb{P}^2$ with the homogenous coordinates $[z_0:z_1:z_2]$, on which we take the Fubini-Study metric $\omega_{\FS}=\ei \del\overline{\del} \log \|z\|^2$. Then, consider the submanifolds $Z_1=\{[z_0:0:z_2] \}\simeq \mathbb{P}^1$ and $Z_2=\{[z_0:z_1:0] \}\simeq \mathbb{P}^1$. We have $Z_1\cap Z_2=\{[1:0:0]\}$.
		Next, we define $f:Z_1\rightarrow Z_2$, $f([z_0:0:z_2])=[z_0:z_2:0]$,
		which is a biholomorphism with one fixed point, $[1:0:0]=:x_0$. 
		Then, we define $$X=\faktor{\mathbb{P}^2}{x \sim f(x)}, \hspace{5pt} \text{for any} \hspace{5pt} x\in Z_1,$$
		and denote $\pi:\mathbb{P}^2\rightarrow X$ the projection.  
		By \ref{Cartan_criterion}, $X$ is a complex space. Indeed, if $x\in Z_1\cup Z_2\setminus \{x_0\}$, then around $\pi(x)$, $X$ is isomorphic as ringed spaces to 
		$$V=\{(v_1,v_2,0)\mid |v_1|<1,|v_2|<1\}\cup \{(v_1,0,v_3)\mid |v_1|<1, |v_3|<1 \}\subset\mathbb{C}^3,$$
		so it is a complex space.
		It remains to study $X$ around $\pi(x_0)$, where it is isomorphic to $\faktor{P}{\sim}$, where $P\subset\mathbb{C}^2$ is the unit polydisc and $(v,0)\sim(0,v)$ for any $|v|<1$. The function $f(v_1,v_2)=v_1+v_2$ is invariant to $\sim$, and $f(v_1,0)\neq f(v_2,0)$ for any $v_1\neq v_2$. Also, the function $g(v_1,v_2)=v_1v_2h(v_1,v_2)$, is invariant to $\sim$ for any holomorphic $h$. For any $(v_1,v_2),(w_1,w_2)\in P$ which are not in the same class, and with at least one of the products $v_1v_2$ and $w_1w_2$ non-zero, we cand find a holomorphic $h$ such that $g(v_1,v_2)\neq g(w_1,w_2)$. 
		
		Furthermore, the metric $\omega_{\FS}$ on $\mathbb{P}^2$ can be defined by strictly psh functions which agree on each $\sim$-class. By \ref{PROP-Coltoiu-Mihalache-s-psh}, these functions descend to a strictly psh functions on $X$, thus $X$ is a compact \K\ space. Since the conditions in \ref{MAIN_THM} are verified by $X$, Vaisman's theorem holds on $X$.

	\end{example}
	
	\vspace{10pt}
	
\begin{example}
	This is in fact a generalization of \cite[Example 4.5]{PS23}, showing a typical scenario for which Vaisman's theorem doesn't hold:
	
	Take a compact \K \ manifold $(M, \omega)$, $\dim_\C M = n$ and let $\tilde{M}$ be the blow-up of $M$ at two distinct points; denote by $Z_1, Z_2 \subset \tilde{M}$ the exceptional divisors, $Z_1 \simeq Z_2 \simeq \mathbb{P}^{n-1}$, and $f: Z_2 \rightarrow Z_1$ a biholomorphism between them. Let $\widetilde{\omega}$ be the \K \ form on $\widetilde{M}$. Now define the complex spaces
	\[
	\tilde{X} = \faktor{\tilde{M} \times \Z}{\sim}, \text{ where } (z, k) \sim (f(z), k + 1) \text{ for any } z \in Z_2, k \in \Z,
	\]        
	and also $X$ obtained by simply glueing $Z_1$ and $Z_2$ via $f$ on $\tilde{M}$. Clearly $\tilde{X}$ is the universal cover of $X$ with Deck group $\Z$. Moreover, we can assume the \K \ metric on $\tilde{M}$ to be any multiple of the Fubini-Study metric $\omega_{FS}$ on the exceptional divisors (for the technical details about the construction of a metric on the blow-up in a point, see \eg \cite[pp. 182-189]{griffithsharris}), so let $\alpha, \beta$ be two metrics such that $\alpha_{\restriction Z_1} = \beta_{\restriction Z_1} = \omega_{FS}$, $\alpha_{\restriction Z_2} = \omega_{FS}$ and $\beta_{\restriction Z_2} = 2\omega_{FS}$. Accordingly we can consider two \K \ structures on $\tilde{X}$:
    \begin{itemize}
	\item $\tilde{\omega}$ such that $\tilde{\omega}_{\tilde{M} \times \{ k \}} = \alpha $ for any $k \in \Z$;
	\item $\tilde{\Omega}$ such that $\tilde{\Omega}_{\tilde{M} \times \{ k \}} = 2^k \beta $ for any $k \in \Z$;
    \end{itemize}
    note that this glueing makes sense using \ref{COR-Coltoiu-Mihalache-s-psh}. Since the metric $\tilde{\omega}$ is just invariant under the action of $\Z$ on $\tilde{X} \to X$, so it descends to a \K \ metric on $X$. On the other hand, $\Z$ acts on $\tilde{\Omega}$ via homotheties, $\gamma_k^* \tilde{\Omega} = 2^k \tilde{\Omega}$, so, by \cite[Theorem 3.10]{PS21}, a metric in the conformal class of $\tilde{\Omega}$ descends to a purely lcK metric on $X$. 
	
\end{example}
	
	


\begin{thebibliography}{100}

            \bibitem[APV23]{APV23} D. Angella, M. Parton, V. Vuletescu: {\em On locally conformally K\" ahler threefolds with algebraic dimension two}, International Mathematics Research Notices \textbf{2023} (5) (2023), pp. 3948-3969.
		
		\bibitem[CM85]{CM85} M. Col\c toiu, N. Mihalache: \textit{Strongly Plurisubharmonic Exhaustion Functions on 1-Convex spaces}, Math. Ann \textbf{270} (1985), pp. 63--68.
		
		\bibitem[Dem12]{Demailly_book} J.-P. Demailly: \textit{Complex analytic and differential geometry},  https://www-fourier.ujf-grenoble.fr/$\sim$demailly/manuscripts/agbook.pdf
		
		\bibitem[Car60]{HCartan60} H. Cartan: \textit{Quotients of complex analytic spaces}, Contributions to function theory (Internat. Colloq. Function Theory, Bombay, 1960), Tata Institute of Fundamental Research, Bombay, 1960, pp. 1--15.
	
		\bibitem[GH78]{griffithsharris} P. Griffiths, J. Harris: \textit{Principles of algebraic geometry}, John Wiley \& Sons (1978).	

        \bibitem[Gr62]{grauert} H. Grauert: {\em \"Uber Modifikationen und exzeptionelle analytische Mengen}, Math. Ann. \textbf{146} (1962), 331--368.
		
		\bibitem[OV24]{OV_book} L. Ornea, M. Verbitsky: {\em Principles of Locally Conformally K\" ahler Geometry}, Birkh\" auser, Progress in Mathematics Vol. 354 (2024).
		
		\bibitem[PS21]{PS21} O. Preda, M. Stanciu: {\em Coverings of locally conformally Kähler complex spaces}, Math. Z. {\bf 298} (2021), pp. 639--651. 
		
		\bibitem[PS23]{PS23} O. Preda, M. Stanciu: {\em Vaisman's theorem for lcK spaces}, Ann. Sc. Norm. Super. Pisa Cl. Sci. (5), Vol. XXIV (2023), pp.2311-2321.
        		
		\bibitem[Vai76]{vaisman1976} I. Vaisman {\em On locally conformal almost K\" ahler manifolds}, Israel J. Math. \textbf{24 No. 3-4} (1976), pp. 338--351.
		
		\bibitem[Vai80]{vaisman1980} I. Vaisman, {\em On locally and globally conformal K\" ahler manifolds}, Trans. Amer. Math. Soc.
		{\bf 262} (1980), pp. 533--542.
		
		\bibitem[V\^ aj96]{vaj96} V. Vâjâitu, {\em K\"ahlerianity of $q$-Stein spaces}, Arch. Math. {\bf{66}} (1996), pp.250--257.

	\end{thebibliography}
\end{document}